\documentclass[11pt,reqno]{amsart}
\usepackage{graphicx,amsmath,amsthm,verbatim,amssymb,enumerate}
\usepackage{xcolor}
\definecolor{darkgreen}{rgb}{0,0.75,0}
\definecolor{darkred}{rgb}{0.75,0,0}
\definecolor{darkmagenta}{rgb}{0.5,0,0.5}

\usepackage[colorlinks,citecolor=darkgreen,linkcolor=darkred,urlcolor=darkmagenta]{hyperref}

\pdfoutput=0 %make is 0 for arxiv submission

\newcommand{\mr}[1]{{\tt \href{http://www.ams.org/mathscinet-getitem?mr=#1}{MR#1}}}
\newcommand{\arxiv}[1]{{\tt \href{http://arxiv.org/abs/#1}{arXiv:#1}}}

\newcommand{\Sett}[2]{\left\{ #1 \; : \; #2 \right\}}

\newcommand{\old}[1]{}

\newcommand{\abs}[1]{{\left\vert\kern-0.25ex #1
    \kern-0.25ex\right\vert}}
\newcommand{\df}{{d_f}} %fractal dimension	
\newcommand{\dw}{{d_w}} %walk dimension
\DeclareRobustCommand{\SkipTocEntry}[5]{}
\newtheorem{theorem}{Theorem}[section]
\newtheorem{prop}[theorem]{Proposition}

\newtheorem{lemma}[theorem]{Lemma}

\theoremstyle{remark}
\newtheorem*{remark}{Remark}

\numberwithin{counter}{section}

\theoremstyle{definition}
\newtheorem{definition}[theorem]{Definition}
\def\L{\mathcal L}

\def\00{\mathbf{0}}

\def\supp{\operatorname{supp}}
\def\F{\mathcal{D}}
\def\E{\mathcal{E}}

\def\N{\mathbb{N}}

\def\R{\mathbb{R}}
\def\EE{\mathbb{E}}
\def\PP{\mathbb{P}}

\newcommand\norm[1]{\left\lVert#1\right\rVert} %norm
\def\CS{\operatorname{CS}}
\def\CSA{\operatorname{CSA}}

\begin{document}
\title{ Davies' method for anomalous diffusions}
\author[M. Murugan]{Mathav Murugan}
\address{Department of Mathematics, University of British Columbia and Pacific Institute for the Mathematical Sciences, Vancouver, BC V6T 1Z2, Canada.}
\email[M. Murugan]{mathav@math.ubc.ca}
%\address{Center for Applied Mathematics, Cornell University, Ithaca, NY 14853, USA}
%\email[M. Murugan]{mkm233@cornell.edu}

\author[L. Saloff-Coste]{Laurent Saloff-Coste$^\dagger$}
\thanks{$\dagger$Both the authors were partially supported by NSF grant DMS 1404435}
\address{Department of Mathematics, Cornell University, Ithaca, NY 14853, USA}
\email[L. Saloff-Coste]{lsc@math.cornell.edu}
\date{\today}
\begin{abstract}
Davies' method of perturbed semigroups is a classical technique to obtain off-diagonal upper bounds on the heat kernel.
However Davies' method does not apply to anomalous diffusions due to the singularity of energy measures.
In this note, we overcome the difficulty by modifying the Davies' perturbation method to obtain sub-Gaussian upper bounds on the heat kernel.
Our computations closely follow the seminal work of Carlen, Kusuoka and Stroock \cite{CKS}. However, a cutoff Sobolev inequality due to Andres and Barlow \cite{AB} is used to bound the energy measure.
\end{abstract}
\maketitle
\section{Introduction}
Davies' method of perturbed semigroups is a well-known method to
obtain off-diagonal upper bounds on the heat kernel. It was
introduced  by E.\ B.\ Davies to obtain the explicit constants in
the exponential term for Gaussian upper bounds \cite{Dav1} using the
logarithmic-Sobolev inequality. Davies' method was extended by
Carlen, Kusuoka and Stroock to a non-local setting \cite[Section 3]{CKS} using
Nash inequality. Moreover, Davies extended this technique to higher order elliptic operators on $\R^n$ \cite[Section 6 and 7]{Dav2}.
More recently Barlow, Grigor'yan and Kumagai
applied Davies' method as presented in \cite{CKS} to obtain
off-diagonal upper bounds for the heat kernel of heavy tailed jump processes
\cite[Section 3]{BGK}.

Despite these triumphs, Davies' perturbation method
has not yet been made to work in the following contexts:
\begin{enumerate}[(a)]
\item Anomalous diffusions (See \cite[Section 4.2]{Bar}).
\item  Jump processes with jump
index greater than or equal to $2$ (See \cite[Remark 1(d)]{MS} and \cite[Section 1]{GHL}).
\end{enumerate}
 The goal of this work is to extend Davies' method to anomalous
diffusions in order to obtain sub-Gaussian upper bounds. In the anomalous
diffusion setting, we use cutoff functions satisfying a cutoff
Sobolev inequality to perturb the corresponding heat semigroup. We use a
recent work of Andres and Barlow \cite{AB} to construct these cutoff
functions. We extend the techniques developed here in a sequel to a non-local setting for the jump processes mentioned in (b) above \cite{MS4}. In \cite{MS4}, we consider the analogue of symmetric stable processes on fractals while in this work we are motivated by Brownian motion on fractals.

Before we proceed, we briefly outline Davies' method as presented in \cite{CKS} and point out  the main difficulty in extending it to the anomalous diffusion setting.
Consider a metric measure space $(M,d,\mu)$ and a Markov semigroup $(P_t)_{t \ge 0}$ symmetric with respect to $\mu$. The most classical case is that of the heat semigroup in $\R^n$ (corresponding to Brownian motion in $\R^n$) associated with the Dirichlet form $\E(f,f)= \int_{\R^n} \abs{\nabla f}^2 \,d\mu$, where $\mu$ is the Lebesgue measure.

Instead of considering the original Markov semigroup $(P_t)_{t \ge 0}$,  we consider the perturbed semigroup
\begin{equation} \label{e-pet}
 \left( P^\psi_t f\right) (x)= e^{\psi(x)} \left( P_t \left( e^{-\psi} f \right) \right)(x)
\end{equation}
where $\psi$ is a `sufficiently nice function'.
Given an ultracontractive estimate  \begin{equation}\label{e-ult}\norm{P_t}_{1 \to \infty} \le m(t)\end{equation} for the diffusion semigroup, Davies' method yields an ultracontractive estimate for the perturbed semigroup
\begin{equation}\label{e-pult}
 \norm{P^\psi_t}_{1 \to \infty} \le m_\psi(t).
\end{equation}
If $p_t(x,y)$ is the kernel of $P_t$,  then the kernel of $P^\psi_t$ is $p_t^\psi(x,y) = e^{-\psi(x)}p_t(x,y) e^{\psi(y)}$.
Therefore by \eqref{e-pult}, we obtain the off-diagonal estimate
\begin{equation} \label{e-offd}
p_t(x,y) \le m_\psi(t) \exp\left( \psi(y)-\psi(x) \right).
\end{equation}
By varying $\psi$ over a class of `nice functions' to minimize the right hand side of \eqref{e-offd}, Davies obtained off-diagonal upper bounds.
In Davies' method as presented in \cite{Dav1,CKS}, it is crucial that the function $\psi$  satisfies
\[
 e^{-2\psi} \Gamma(e^\psi,e^\psi) \ll \mu \hspace{0.5cm} \mbox{and} \hspace{0.5cm}  e^{2\psi} \Gamma(e^{-\psi},e^{-\psi}) \ll \mu,
\]
where $\Gamma(\cdot,\cdot)$ denotes the corresponding energy measure (cf. Definition \ref{d-energy}). For the classical example of heat semigroup in $\R^n$ described above, the energy measure $\Gamma(f,g)$ is $\nabla f. \nabla g \,d\mu$, where $\mu$ is the Lebesgue measure.

 In fact the expression of $m_\psi$ in \eqref{e-pult} depends on the uniform bound on the Radon-Nikodym derivatives of the energy measure given by (See \cite[Theorem 3.25]{CKS})
 \[ \Gamma(\psi) := \norm{\frac{d e^{-2\psi} \Gamma(e^\psi,e^\psi) }{d \mu}}_\infty \vee\norm{\frac{d  e^{2\psi} \Gamma(e^{-\psi},e^{-\psi}) }{d\mu}}_\infty .\]

The main difficulty  in extending Davies' method to anomalous diffusions is that,  for many `typical fractals' that satisfy a sub-Gaussian estimate,
the energy measure $\Gamma( \cdot, \cdot)$
is singular with respect to the underlying symmetric measure $\mu$ \cite{Hin,Kus,BST}.
This difficulty is well-known to experts (for instance, \cite[p. 1507]{BB} or \cite[p. 86]{Kum}).
In this context, the condition $e^{-2\psi} \Gamma(e^\psi,e^\psi) \ll \mu$ implies that $\psi$ is
necessarily a constant,  in which case the off-diagonal estimate of \eqref{e-offd} is not an improvement over the diagonal estimate \eqref{e-pult}.
%It is easier to illustrate the main ideas in a local setting.
%The corresponding non-local theory can be will be developed elsewhere.

We briefly recall some fundamental notions regarding Dirichlet form and refer the reader to \cite{FOT} for details.
Let $(M,d,\mu)$ be a locally compact metric measure space where $\mu$ is a positive Radon measure on $M$ with $\supp(\mu)=M$.  We denote by $\langle \cdot, \cdot \rangle$ the inner product on $L^2(M,\mu)$.
Let $ \mathbf{X}= (\Omega, \mathcal{F}_\infty,\mathcal{F}_t,X_t, \PP_x)$ denote the diffusion corresponding to a strongly, local regular Dirichlet form. Here $\Omega$ denotes the totality of right continuous paths with left-limits from $[0,\infty)$ to $M$ and $\PP_x$ denotes the law of the process conditioned to start  at $X_0=x$.
The corresponding Markov semigroup $\Sett{P_t} {t \ge 0}$ of $\mathbf{X}$ is defined by
$P_tf(x):= \EE_x[f(X_t)]$, where $\EE_x$ denote the expectation with respect to the measure $\PP_x$.
These
operators  $\Sett{P_t} {t \ge 0}$ form a strongly continuous semigroup of self-adjoint contractions. The
Dirichlet form $(\E,\F)$ associated with $\mathbf{X}$ is the symmetric, bilinear form
\[ \E(u,v):= \lim_{t \downarrow 0} \frac{1}{t}\langle u - P_t u, v \rangle
\]
defined on the domain
\[\F  := \Sett{u \in L^2(M,\mu)}{ \sup_{t >0}
 \frac{1}{t}\langle u - P_t u,u\rangle < \infty}.\]

Recall that a Dirichlet form $(\E,\F)$ on $L^2(M,\mu)$ is said to be \emph{regular} if $C_c(M) \cap \F$ is dense in both $(C_c(M), \norm{\cdot}_\infty)$ and the Hilbert space $(\F,\E_1)$.
Here $C_c(M)$ is the space of continuous functions with compact
support in $M$ and $\E_1(\cdot, \cdot):= \E(\cdot, \cdot)+ \langle \cdot,\cdot \rangle$ denotes the inner product on $\F$. For a $\mu$-measurable function $u$ let  $\operatorname{Supp}[u]$ denote the support of the measure $u \,d\mu$.
 We say that a Dirichlet form $(\E,\F)$ on $L^2(M,\mu)$ to be \emph{strongly local} if it satisfies the following property: For all functions $u,v \in \F$ such that $\operatorname{Supp}[u]$,  $\operatorname{Supp}[u]$ are compact and $v$ is constant on a neighbourhood of $\operatorname{Supp}[v]$, we have $\E(u,v)=0$. For example, the form corresponding to the heat semigroup on $\R^n$ defined by $(f \mapsto \int_{\R^n} \abs{\nabla f}^2 \,d\mu, W^{1,2}(\R^n))$ is a regular, strongly local Dirichlet form on $L^2(\R^n,\mu)$, where $\mu$ is the Lebesgue measure and $W^{1,2}$ denotes the Sobolev space.

We denote by $B(x,r):= \{y \in M: d(x,y)<r \}$  the ball centered at $x$ with radius $r$ and by
\[
 V(x,r):= \mu(B(x,r))
\]
the corresponding volume. We assume that the metric measure space is Ahlfors-regular: meaning that there exist $C_1>0$ and $\df >0$ such that
\begin{equation} \label{ahl}\tag*{$\operatorname{V}(\df)$}
 C_1^{-1} r^\df \le V(x,r) \le C_1 r^\df
\end{equation}
for all $x \in M$ and for all $r \ge 0$. The quantity $\df >0$ is called the \emph{volume growth exponent} or \emph{fractal dimension}.
Let  $p_t(\cdot,\cdot)$ be the (regularised) kernel  of $P_t$ with respect to $\mu$ \cite[eq.~(1.10)]{AB}.
We are interested in obtaining sub-Gaussian upper bounds of the form
\begin{equation} \tag*{$\operatorname{USG}(\df,\dw)$}
\label{usg} p_t(x,y) \le \frac{C_1}{t^{\df/\dw}} \exp \left( -C_2 \left( \frac{d(x,y)^\dw}{t}\right)^{1/(\dw-1)} \right)
\end{equation}
where $\dw \ge 2$ is the \emph{escape time exponent} or \emph{walk dimension}.
It is known that if the heat kernel $p_t$ satisfies \ref{usg}, then $d_w \ge 2$ (cf. \cite[p. 252]{Hin2}). The corresponding diffusion $X_t$ then has a diffusive speed of at least $t^{1/d_w}$ (up to constants).  This means that a process starting at $x$ first exits a ball $B(x,r)$ at the time $\tau_{B(x,r)} \gtrsim r^{d_w}$ (cf. \cite[Lemma 5.3]{AB}). Moreover, if the process satisfies a matching sub-Gaussian lower bound for $p_t$ with different constants, then $\tau_{B(x,r)} \asymp r^{d_w}$. For comparison,  recall that the Brownian motion on Euclidean space has a Gaussian heat kernel and satisfies $\tau_{B(x,r)} \asymp r^{2}$.

Such sub-Gaussian estimates  are typical of many fractals (cf. \cite[Theorem 8.18]{Bar0}).
We assume the on-diagonal bound corresponding to the sub-Gaussian estimate of \ref{usg}. That is, we assume that there exists $C_1>0$ such that
\begin{equation}
\label{e-diag} p_t(x,x) \le \frac{C_1}{t^{\df/\dw}}
\end{equation}
for all $x \in M$ and for all $t>0$.
The on-diagonal estimate of \eqref{e-diag} is equivalent to the following Nash inequality (\cite[Theorem 2.1]{CKS}): there exists $C_N>0$ such that
\begin{equation} \label{nash} \tag*{$\operatorname{N}(\df,\dw)$}
\norm{f}_2^{2(1+ \dw/\df)} \le C_N \E(f,f) \norm{f}_1^{2\dw/\df}
\end{equation}
for all $f \in \F \cap L^1(M,\mu)$. The Nash inequality \ref{nash} may be replaced by an equivalent Sobolev inequality, a logarithmic Sobolev inequality or  a Faber-Krahn inequality (See \cite{BCLS}).
However, we will follow the approach of \cite{CKS} and use the Nash inequality. Such a Nash inequality can be obtained from geometric assumptions like a Poincar\'{e} inequality and a volume growth assumption like \ref{ahl}.

Since $\E$ is regular, it follows that $\E(f,g)$ can be written in terms of a signed measure $\Gamma(f,g)$ as
\[
 \E(f,g) = \int_M \Gamma (f,g),
\]
where the energy measure $\Gamma$ is defined as follows.
\begin{definition}\label{d-energy}
For any essentially bounded $f \in \F$, $\Gamma(f,f)$ is the unique Borel measure on $M$ (called the energy measure) on $M$ satisfying
\[
 \int_M g \,d\Gamma(f,f) = \E(f,fg) -\frac{1}{2} \E(f^2,g)
\]
for all essentially bounded $g \in \F \cap C_c(M)$; $\Gamma(f,g)$ is then defined by polarization.
\end{definition}
We shall use the following properties of the energy measure.
\begin{enumerate}
 \item[(i)] \emph{Locality}: For all functions $f,g \in \F$ and all measurable sets $G \subset M$ on which $f$ is constant
 \[
  \mathbf{1}_G d\Gamma(f,g) = 0
 \]
 \item[(ii)] \emph{Leibniz and chain rules}: For $f,g \in \F$ essentially bounded and $\phi \in C^1(\R)$,
 \begin{eqnarray}
 \label{e-lei} d\Gamma(fg,h) &=& f d\Gamma(g,h)+g d\Gamma(f,h)\\
 \label{e-cha} f \Gamma(\phi(f),g) &=& \phi'(f) d\Gamma(f,g).
 \end{eqnarray}
\end{enumerate}
We wish to obtain an off-diagonal estimate using Davies' perturbation method.
The main difference from the previous implementations of the method is that,
in addition to an on-diagonal upper bound (or equivalently Nash inequality), we also require a cutoff Sobolev inequality.
Spaces satisfying the sub-Gaussian upper bound given in \ref{usg}
 necessarily satisfy the cutoff Sobolev annulus inequality \ref{csa},
 a condition introduced by Andres and Barlow \cite{AB}.
The condition $\CSA$ simplifies the cut-off Sobolev inequalities $\CS$ which were originally introduced by Barlow and Bass \cite{BB} for weighted graphs.
The significance of the cut-off Sobolev inequalities $\CS$ and $\CSA$ is that they are stable under bounded perturbations of the Dirichlet form (Cf. \cite[Corollary 5.2]{AB}).
Moreover, the condition $\CS$ is stable under quasi-isometries (rough isometries) of the underlying space \cite[Theorem 2.21(b)]{BBK}.
Therefore cutoff Sobolev inequalities provide a robust method to obtain heat kernel estimates with anomalous time-space scaling. We now define the cutoff Sobolev inequality \ref{csa}.
\begin{definition}
 Let $U \subset V$ be open sets in $M$ with $U \subset \bar{U} \subset V$. We say that a continuous function $\phi$ is a cutoff function for $U \subset V$ if  $\phi \equiv 1$ on
 $U$ and $\phi \equiv 0$ on $V^c$.
\end{definition}
\begin{definition} (\cite[Definition 1.10]{AB}) \label{d-csa}
We say \ref{csa} holds if there exists $C_1,C_2>0$ such that for every $x \in M$, $R > 0$, $r>0$, there exists a cutoff function $\phi$ for  $B(x,R) \subset B(x,R+r)$ such that if $f \in \F$, then
\begin{equation} \label{csa}  \tag*{$\operatorname{CSA}(\dw)$}
 \int_U f^2 \, d \Gamma(\phi,\phi) \le C_1 \int_U \phi^2 \, d\Gamma(f,f) + \frac{C_2}{r^\dw} \int_U f^2 \, d\mu,
\end{equation}
where $U= B(x,R+r) \setminus \overline{B(x,r)}$.
\end{definition}
It is clear that the condition $\CSA(\dw)$ is preserved by bounded perturbations of the Dirichlet form.
The above definition is slightly different to the one introduced in \cite[Definition 1.10]{AB}, where the constant  $C_1$ is taken to be $1/8$. However both definitions are equivalent due to a `self-improving' property of $\CSA(\dw)$ \cite[Lemma 5.1]{AB}.

Our main result is that the Nash inequality \ref{nash} and the cutoff Sobolev inequality \ref{csa} imply the desired sub-Gaussian estimate \ref{usg}.
By  \cite[Theorem
1.12]{AB}, it is known that both \ref{nash} and  \ref{csa} are also necessary for 
the  sub-Gaussian estimate \ref{usg} to hold .
 More precisely,
\begin{theorem} \label{t-main}
 Let $(M,d,\mu)$ be a locally compact metric measure space that satisfies \ref{ahl} with volume growth exponent $\df$.
 Let $(\E,\F)$ be a strongly local, regular, Dirichlet form whose energy measure $\Gamma$ satisfies the cutoff Sobolev inequality \ref{csa} for some $d_w \ge 2$.
 Then the Nash inequality \ref{nash} implies the sub-Gaussian upper bound \ref{usg}.
\end{theorem}
\begin{remark}
 The above properties given by \ref{ahl} and \ref{usg} are a special case of the more general assumptions of volume doubling and heat kernel upper bounds with a general time-space scaling of \cite{AB}.
In fact, Theorem \ref{t-main} is subsumed by \cite[Theorem
1.12]{AB}. A recent work of Lierl provides an alternate proof of the sub-Gaussian estimates in \cite{AB} using Moser's iteration method and extends the results to certain time-dependent, non-symmetric local bilinear forms  \cite{Lie}. Like earlier work by Andres and Barlow and
the present work, Lierl's arguments involve improved control on some cutoff functions.

Our methods give an alternate proof to
\cite[Theorem 1.12]{AB} in a restricted setting. Moreover we show in \cite{MS4} that this techniques can be adapted to the non-local setting to provide new result and
 resolve the conjecture posed in \cite[Remark 1(d)]{MS}.
\end{remark}

\section{Off diagonal estimates using Davies' method}\label{s-Davies}
Spaces satisfying \ref{csa} have a rich class of cutoff functions with low energy. We start by studying energy estimates of these cutoff functions.
\subsection{Self-improving property of $\CSA$}\label{ss-self}
The cutoff Sobolev inequality \ref{csa} has a self-improving property which states that the constants $C_1,C_2$ in \ref{csa} are flexible. For example, we can decrease the value of $C_1$ in \ref{csa}
by increasing $C_2$ appropriately. This is quantified in Lemma \ref{l-sip}. Lemma \ref{l-sip} is essentially contained in \cite{AB}; we simplify the proof and obtain a slightly stronger result. 
\begin{lemma}\label{l-sip} Let $(M,d,\mu)$ satisfy \ref{ahl}. Let $(\E,\F)$ denote a strongly local, regular, Dirichlet form with energy measure $\Gamma$ that satisfies \ref{csa}.
There exists $C_1,C_2>0$ such that
for all $\rho \in (0,1]$,  for all $R,r >0$ and for all $x \in M$,  there exists a  cutoff function $\phi_\rho = \phi_\rho(f)$ for $B(x,R) \subset B(x,R+r)$ that satisfies
 \begin{equation} \label{e-sip}
 \int_U f^2 \, d \Gamma(\phi_\rho,\phi_\rho) \le 4 C_1 \rho^2  \int_U  d\Gamma(f,f) + \frac{C_2 \rho^{2 -\dw} }{r^\dw} \int_U f^2 \, d\mu
\end{equation}
for all $f \in \F$,
 where $C_1,C_2$ are the constants in \ref{csa}. %
 Further the cutoff function $\phi_\rho$ above satisfies
\begin{equation} \label{e-siv}
\abs{ \phi_\rho(y)   - \left( \frac{R+r - d(x,y)}{r} \right) } \le 2 \rho
\end{equation}
for all $y \in B(x,R+r)\setminus B(x,R)$.
\end{lemma}
\begin{remark}
Lemma \ref{l-sip} is essentially contained in work of Andres and Barlow \cite[Lemma 5.1]{AB}. More recently, following \cite[Lemma 5.1]{AB}, J. Lierl obtained a cutoff Sobolev inequality \cite[Lemma 2.3]{Lie} that is similar to Lemma \ref{l-sip}. However, the estimate \eqref{e-siv} is new and it shows that the cutoff functions converge in $L^\infty$ norm as $\rho \to 0$ to the ``linear cutoff function''. The constructions in \cite[Lemma 5.1]{AB} and \cite[Lemma 2.3]{Lie} converge in $L^\infty$ norm as $\rho \to 0$ to a somewhat more complicated cutoff function that depends on $\dw$. The proof below was suggested to us by Martin Barlow.
\end{remark}
\begin{proof}[Proof of Lemma \ref{l-sip}]
Let $x \in M$, $r>0$, $R>0$, $\rho>0$, $f \in \F$. Define $n:= \lfloor \rho^{-1} \rfloor \in  [\rho^{-1}/2 ,\rho^{-1}]$.  We divide the annulus $U= B(x,R+r) \setminus B(x,R)$ into $n$-annuli $U_1,U_2,\ldots,U_n$ of equal width, where
\[
U_i:= B(x,R+  ir/n) \setminus  B(x,R+  (i-1)r/n), \hspace{5mm} i=1,2,\ldots,n.
\]
By \ref{csa}, there exists a cutoff function $\phi_i$ for $B(x,R+(i-1)r/) \subset B(x,R+ir/n)$ satisfying
\begin{equation}\label{e-sip1}
\int_{U_i} f^2 \, d\Gamma(\phi_i,\phi_i) \le C_1 \int_{U_i} \, d\Gamma(f,f) + \frac{C_2}{ (r/n)^\dw} \int_{U_i} f^2 \, d\mu
\end{equation}
for $i=1,2,\ldots,n$.  We define $\phi= n^{-1} \sum_{i=1}^n \phi_i$. By locality, we have
\begin{equation}
\label{e-sip2} d\Gamma(\phi,\phi) = \frac{1}{n^2}\sum_{i=1}^n d\Gamma(\phi_i,\phi_i).
\end{equation}
Therefore by \eqref{e-sip2},  \eqref{e-sip1} and $\rho^{-1}/2 \le  n= \lfloor \rho^{-1} \rfloor \le \rho^{-1}$ , we obtain
\begin{align}
\int_U f^2 \,d\Gamma(\phi,\phi) &= n^{-2} \sum_{i=1}^n   \int_U f^2 \,d\Gamma(\phi,\phi) \nonumber \\
& \le n^{-2} \sum_{i=1}^n \left(  C_1 \int_{U_i}  d\Gamma(f,f) + \frac{C_2}{ (r/n)^\dw} \int_{U_i} f^2 \, d\mu \right) \nonumber \\
& \le  C_1 n^{-2} \int_{U}   d\Gamma(f,f) + \frac{C_2 n^{\dw-2} }{  r^\dw } \int_{U_i} f^2 \, d\mu \nonumber \\
& \le  4 C_1 \rho^2 \int_{U}   d\Gamma(f,f) + \frac{C_2 \rho^{2-\dw} }{  r^\dw } \int_{U_i} f^2 \, d\mu. \nonumber
\end{align}
This completes the proof of \eqref{e-sip}.

Note that if $y \in U_i$, then $ 1 - i/n \le \phi(y) \le 1 - (i-1)/n$ and $R +  (i-1)r/n \le d(x,y) < R +  ir/n$, for each $1 \le i \le n$. This along with $n^{-1} \le 2 \rho$ implies \eqref{e-siv}.

\end{proof}

Observe that by \eqref{e-siv}, the cutoff function $\phi_\rho$ for $B(x,r)\subset B(x,R+r)$ satisfies
\[
 \lim_{\rho \downarrow 0} \phi_\rho(y) =  1 \wedge \left(\frac{(R+r-d(x,y))_+}{r} \right).
\]
\subsection{Estimates on perturbed forms}\label{ss-perturbed}
The key to carry out Davies' method is the following elementary inequality.
\begin{lemma} \label{l-key}
 Let $(\E,\F)$ be a strongly local, regular, Dirichlet form. Then
 \begin{equation} \label{e-key}
  \E(e^\psi f^{2p-1} , e^{-\psi} f) \ge \frac{1}{p} \E(f^p,f^p) - p \int_{M} f^{2p} \, d\Gamma (\psi,\psi)
 \end{equation}
 for all $f \in \F$, $\psi \in \F$ and $p \in [1,\infty)$.
\end{lemma}
\begin{proof}
 Using Leibniz rule \eqref{e-lei} and chain rule \eqref{e-cha}, we obtain
 \begin{eqnarray} \label{e-ky1}
\lefteqn{\Gamma(e^\psi f^{2p-1} , e^{-\psi} f) - \frac{1}{p} \Gamma(f^p,f^p) - p  f^{2p} \Gamma (\psi,\psi)} \nonumber \\
&=& (p-1) \left( f^{2(p-1)} \Gamma(f,f) + f^{2p} \Gamma(\psi,\psi) - 2 f^{2p-1}\Gamma(f,\psi) \right).
 \end{eqnarray}
 By \cite[Theorem 3.7]{CKS} and the Cauchy-Schwarz inequality,  we have
 \begin{equation*}
   \int_M f^{2p-1}\, d\Gamma(f,\psi) \le \left( \int_M f^{2(p-1)} \, d  \Gamma(f,f) \cdot \int_M  f^{2p} \, d\Gamma(\psi,\psi) \right)^{1/2}.
 \end{equation*}
Therefore
 \begin{equation} \label{e-ky2}
  2 \int_M f^{2p-1}\, d\Gamma(f,\psi) \le  \int_M f^{2(p-1)} \, d  \Gamma(f,f) +  \int_M  f^{2p} \, d\Gamma(\psi,\psi).
 \end{equation}
 By integrating \eqref{e-ky1} and using \eqref{e-ky2}, we obtain \eqref{e-key}.
\end{proof}
Davies used the bound
\[
\int_{M} f^{2p} \, d\Gamma (\psi,\psi) \le \norm{\frac{d \Gamma(\psi,\psi)}{d \mu}}_\infty \norm{f}_{2p}^{2p}
\]
to control a term in \eqref{e-key}. However for anomalous diffusions, the energy measure is singular to $\mu$.
We will instead use \ref{csa} to bound $\int_{M} f^{2p} \, d\Gamma (\psi,\psi) $ by choosing $\psi$ to be a multiple of the cutoff function satisfying \ref{csa}.
The following estimate is analogous to \cite[Theorem 3.9]{CKS} but unlike in \cite{CKS},
the cutoff functions depend on both $p$ and $\lambda$. This raises new difficulties in the implementation of Davies' method.
\begin{prop} \label{p-drive}
Let  $(M,d,\mu)$ be a metric measure space. Let $(\E,\F)$ be a strongly local, regular, Dirichlet form on $M$ satisfying \ref{csa}. There exists $C>0$ such that, for all $\lambda \ge 1$, for all $r>0$, for all $x \in M$, and
for all $p \in [1,\infty)$, there exists a cutoff function $\phi=\phi_{p,\lambda}$ on $B(x,r) \subset B(x,2r)$  such that
\begin{equation} \label{e-dri3}
  \E(e^{\lambda \phi} f^{2p-1}, e^{-\lambda \phi} f) \ge  \frac{1}{2p} \E(f^p,f^p)  - C \frac{\lambda ^\dw p^{\dw-1}}{r^\dw} \norm{f}_{2p}^{2p},
\end{equation}
for all $f \in \F$.
 There exists $C'>0$ such that the cutoff functions $\phi_{p,\lambda}$ above satisfy
\begin{equation} \label{e-step}
\norm{\exp \left( \lambda (\phi_{p,\lambda} - \phi_{2p,\lambda})  \right)}_\infty \vee \norm{\exp \left( -\lambda (\phi_{p,\lambda} - \phi_{2p,\lambda})  \right)}_\infty \le \exp(C'/p)
\end{equation}
for all $\lambda \ge 1$ and for all $p \ge 1$.
\end{prop}
\begin{proof}
 This theorem follows from Lemma \ref{l-key} and Lemma \ref{l-sip}. Let $x \in M$ and $r>0$ be arbitrary.
 Using \eqref{e-key}, we obtain
 \begin{equation} \label{e-ts1}
    \E(e^{\lambda \phi} f^{2p-1}, e^{-\lambda \phi} f) \ge \frac{1}{p} \left( \E(f^p,f^p) - (p \lambda)^2 \int_M f^{2p} \, d\Gamma(\phi,\phi) \right).
 \end{equation}
By Lemma \ref{l-sip} and fixing $\rho^2= (p\lambda)^{-2}/(8C_1)$ in \eqref{e-sip}, we obtain a cutoff function $\phi= \phi_{p,\lambda}$ for $B(x,r)\subset B(x,2r)$ and $C > 0$ such that
\begin{equation}\label{e-ts2}
 (p \lambda)^2 \int_M f^{2p} \, d\Gamma(\phi,\phi) \le \frac{1}{2} \E(f^p,f^p) + C \frac{ (\lambda p)^\dw}{r^\dw} \int_M f^{2p}\,d\mu.
\end{equation}
By \eqref{e-ts1} and \eqref{e-ts2}, we obtain \eqref{e-dri3}.

By \eqref{e-siv} and the above calculations, there exists $C'>0$ such that the cutoff functions $\phi_{p,\lambda}=\phi_{p,\lambda}(f)$ satisfy
\[
 \norm{ \phi_{p,\lambda} - \phi_{2p,\lambda} }_\infty \le \frac{C'}{p \lambda}
\]
for all $p \ge 1$, for all $\lambda \ge 1$, for all $f \in \F$, for all $x \in M$ and for all $r>0$.
This immediately  implies \eqref{e-step}.

\end{proof}
\begin{remark}
 Estimates similar to \eqref{e-dri3}, were introduced by Davies in \cite[equation (3)]{Dav2} to obtain off-diagonal estimates for  higher order  (order greater than 2) elliptic operators.
 Roughly speaking, the generator $\L$ for anomalous diffusion with walk dimension $\dw$ behaves like an `elliptic operator of order $\dw$'.
 However the theory presented in \cite{Dav2} is complete only when the `order' $\dw$ is bigger than the volume growth exponent $\df$, \emph{i.e.} in the strongly recurrent case.
 This is because the method in \cite{Dav2} relies on a Gagliardo-Nirenberg inequality which is true only in the strongly recurrent setting. We believe that one can adapt the methods of \cite{Dav2}
 to obtain an easier proof for the strongly recurrent case. However, we will not impose any such restrictions and our proof will closely follow the one in \cite{CKS}.
\end{remark}

\subsection{ Proof of Theorem \ref{t-main}:}
Let $\lambda \ge 1$ and $x \in M$ and $r>0$. Let $p_k=2^k$ and let
$\psi_k = \lambda \phi_{p_k,\lambda}$, where $ \phi_{p_k,\lambda}$
is a cutoff function on $B(x,r) \subset B(x,2r)$  given by
Proposition \ref{p-drive}. We write
\begin{equation} \label{e-fdef}
 f_{t,k}:= P_t^{\psi_{k} } f
\end{equation}
for all $k \in \N$, where $f \in \F$ and $P_t^{\psi_{k} }$ denotes the perturbed semigroup as in \eqref{e-pet}.

Using \eqref{e-dri3}, there exists $C_0>0$ such that
\begin{eqnarray}
\nonumber \frac{d}{dt} \norm{f_{t,0}}_{2}^2 &=& -2 \E\left(e^{\psi_1} f_{t,0}, e^{- \psi_1 } f_{t,0}\right) \\
&\le& 2 C_0 \frac{ \lambda^\dw }{r^\dw} \norm{f_{t,0}}_2^2  \label{e-sr1}
\end{eqnarray}
and
\begin{eqnarray}
\nonumber \frac{d}{dt} \norm{f_{t,k}}_{2p_{k}}^{2p_k} &=& -2 p_k \E\left(e^{\psi_k}  f_{t,k}^{2p_k-1}, e^{-\psi_k }  f_{t,k} \right) \\
&\le& - \E \left( f_{t,k}^{p_k}, f_{t,k}^{p_k}\right) + 2 C_0 \left(\frac{\lambda p_k}{r} \right)^\dw \norm{f_{t,k}}_{2 p_k}^{2 p_k} \label{e-sr2}
\end{eqnarray}
for all $k \in \N^*$.
By \eqref{e-sr1}, we obtain
\begin{equation}
 \label{e-sr3} \norm{f_{t,0}}_{p_1} =\norm{f_{t,0}}_2 \le \exp\left(  C_0 \lambda^\dw t /r^\dw \right) \norm{f}_2.
\end{equation}
Using \eqref{e-sr2} and the Nash inequality \ref{nash}, we obtain
\begin{equation}
\label{e-sr4} \frac{d}{dt} \norm{f_{t,k}}_{2p_{k}} \le - \frac{1}{2 C_N p_k} \norm{f_{t,k}}_{2 p_k}^{1+ 2 \dw p_k /\df }\norm{f_{t,k}}_{p_k}^{-2 \dw p_k/\df} + C_0 p_k^{\dw-1}\left( \frac{\lambda}{r}\right)^\dw\norm{f_{t,k}}_{2p_k}
\end{equation}
for all $k \in \N^*$. By \eqref{e-step} and the fact that $P_t$ is a contraction on $L^\infty$, we have
\begin{equation} \label{e-sc}
  \exp(-2 C_1/p_k) f_{t,k+1}  \le f_{t,k} \le \exp(2 C_1/p_k) f_{t,k+1}
\end{equation}
for all $k  \in \N_{\ge 0}$.
Combining \eqref{e-sr4} and \eqref{e-sc}, we obtain
\begin{equation}
 \label{e-sr5}  \frac{d}{dt} \norm{f_{t,k}}_{2p_{k}} \le - \frac{1}{C_A p_k} \norm{f_{t,k}}_{2 p_k}^{1+ 2 \dw p_k /\df }\norm{f_{t,k-1}}_{p_k}^{-2 \dw p_k/\df} + C_0 p_k^{\dw-1}\left( \frac{\lambda}{r}\right)^\dw\norm{f_{t,k}}_{2p_k}
\end{equation}
for all $k \in \N^*$, where $C_A= 2 C_N \exp ( 8\dw C_1 /\df)$.

To obtain off-diagonal estimates using the differential inequalities \eqref{e-sr5} we use the following lemma.
The following lemma is analogous to \cite[Lemma 3.21]{CKS} but the statement and its proof is slightly modified to suit our anomalous diffusion context with walk dimension $\dw$.

\begin{lemma}\label{l-dife}
Let $w:[0,\infty) \to (0,\infty)$ be a  non-decreasing function and suppose that $u \in C^1([0,\infty); (0,\infty))$ satisfies
\begin{equation} \label{e-cks}
 u'(t) \le - \frac{\epsilon}{p} \left( \frac{t^{(p-2)/\theta p}}{w(t)}\right)^{\theta p } u^{1+\theta p}(t) + \delta p^{\dw - 1}u(t)
\end{equation}
for some positive $\epsilon, \theta$ and $\delta$, $\dw \in [2,\infty)$ and $p=2^k$ for some $k \in \N^*$. Then  $u$ satisfies
\begin{equation} \label{e-fs} %fs= Fabes-Stroock
 u(t) \le  \left( \frac{2 p^{\dw} }{\epsilon \theta} \right)^{1/\theta p} t^{(1-p)/\theta p} w(t)  e^{\delta t/p}
\end{equation}
\end{lemma}
\begin{proof}
 Set $v(t) = e^{- \delta p^{\dw-1}t} u(t)$. By \eqref{e-cks}, we have
 \[
  v'(t)= e^{- \delta p^{\dw-1}t} \left( u'(t) - \delta p^{\dw -1} u(t) \right) \le - \frac{\epsilon t^{p-2}}{p w(t)^{\theta p}} e^{\theta \delta p^{\dw}t} v(t)^{1+\theta p}.
 \]
Hence
\[
 \frac{d}{dt} \left( v(t) \right)^{-\theta p} \ge \epsilon \theta t^{p-2} w(t)^{-\theta p} e^{\theta \delta p^{\dw}t}
\]
and so, since $w$ is non-decreasing
\begin{equation} \label{e-fs1}
 e^{ \delta \theta p^\dw t} u(t)^{-\theta p} \ge \epsilon \theta w(t)^{-\theta p} \int_0^t s^{(p-2)}  e^{\theta \delta p^{\dw}s} \, ds.
\end{equation}
Note that
\begin{eqnarray}
\nonumber \int_0^t s^{(p-2)}  e^{\theta \delta p^{\dw}s} \, ds &\ge& (t/\delta \theta p^\dw)^{p-1} \int_{\delta \theta p^\dw(1- 1/p^\dw)}^{\delta \theta p^\dw} y^{(p-2)} e^{ty}\, dy \\
\nonumber &\ge & \frac{t^{p-1}}{p-1} \exp \left( \delta \theta p^\dw t -  \delta \theta t \right) \left[ 1 - (1- p^{-\dw})^{p-1} \right] \\
 \label{e-fs2} &\ge & \frac{t^{p-1}}{2p^{\dw}} \exp \left( \delta \theta p^\dw t - \delta \theta  t \right).
\end{eqnarray}
In the last line above, we used the bound $(1-p^{-\dw})^{p-1} \ge 1 - p^{-\dw}(p-1)$ for all $p,\dw \ge 2$.
Combining \eqref{e-fs1} and \eqref{e-fs2} yields \eqref{e-fs}.
\end{proof}
We now pick $f \in L^2(M,\mu)$ and $f \ge 0$ with $\norm{f}_2=1$.
Let $u_{k}(t)= \norm{f_{t,k-1}}_{p_{k}}$  and  let  \[w_k(t)= \sup \{ s^{ \df(p_k-2)/(2 \dw p_k)} u_k(s) : s \in (0,t] \}.\]
By \eqref{e-sr3}, $w_1(t) \le \exp( C_0 \lambda^\dw t /r^\dw)$. Further by \eqref{e-sr5}, $u_{k+1}$ satisfies \eqref{e-cks} with
$\epsilon = 1/C_A$, $\theta = 2 \dw/\df$, $\delta = C_0 (\lambda/r)^\dw$, $w=w_k$ and $p=p_k$. Hence by \eqref{e-fs},
\[
 u_{k+1}(t) \le ( 2^{\dw k +1}/\epsilon \theta)^{1/(\theta p_k)} t^{(1-p_k)/\theta p_k} e^{\delta t/p_k} w_k(t).
\]
Therefore
\[
 w_{k+1}(t)/w_k(t) \le ( 2^{\dw k +1}/\epsilon \theta)^{1/(\theta 2^k)} t^{-1/(2^k \theta)} e^{\delta t/2^k}
\]
for $k \in \N^*$. Hence, we obtain
\[
 \lim_{k \to \infty} w_{k}(t) \le   C_2 t^{-1/\theta} e^{\delta t}   w_1(t) \le C_2 t^{-1/\theta}  \exp( C_0 \lambda^ \dw t/r^\dw )
\]
where $C_2 =C_2( \dw, \epsilon, \theta)$.
Since $P_t$ is a contraction on all $L^p(M,\mu)$ for $1\le p \le \infty$, we obtain
\[
 \lim_{k \to \infty} u_k(t) = \norm{ P_t^{\psi_\infty}f}_\infty \le  \frac{C_2}{t^{\df/2\dw}} \exp(C_0 \lambda^ \dw t/r^\dw ).
\]
where $\psi_\infty=\lim_{k \to \infty} \psi_k$.
Since the above bound holds for all $f \in L^2(M,\mu)$ with $\norm{f}_2=1$, we have
\[
 \norm{P_t^{\psi_\infty}  }_{2 \to \infty} \le  \frac{C_2}{t^{\df/2\dw}} \exp( C_0 \lambda^ \dw t/r^\dw ).
\]
The estimate is unchanged if we replace $\psi_k$'s by $-\psi_k$. Since $P_t^{-\psi}$ is the adjoint of $P_t^\psi$, by duality we have that
\[
 \norm{P_t^{\psi_\infty}}_{1\to 2} \le  \frac{C_2}{t^{\df/2\dw}} \exp( C_0 \lambda^ \dw t/r^\dw ).
\]
Combining the above, we have
\begin{equation}
 \norm{P_t^{\psi_\infty} }_{1\to \infty} \le  \norm{P_{t/2}^{\psi_\infty} }_{1\to 2}  \norm{P_{t/2}^{\psi_\infty} }_{2 \to \infty} \le \frac{C_2 2^{\df/\dw}}{t^{\df/\dw}} \exp( C_0 \lambda^ \dw t/r^\dw ).
\end{equation}
Therefore
\[
 p_t(x,y) \le  \frac{C_2 2^{\df/\dw}}{t^{\df/\dw}} \exp( C_0 \lambda^ \dw t/r^\dw + \psi_\infty(y) - \psi_\infty(x) ).
\]
for all $x,y \in M$ and for all $r,t>0$ and $\lambda \ge 1$.
If we choose $r=d(x,y)/2$, we have  $\psi_\infty(y) - \psi_\infty(x)= -\lambda$. This yields
\[
 p_t(x,y) \le  \frac{C_3}{t^{\df/\dw}} \exp(C_4 \lambda^ \dw t/d(x,y)^\dw - \lambda ).
\]
where $C_3,C_4>1$.  Assume $\lambda = C_4^{-1/(\dw-1)} (d(x,y)^\dw/t)^{1/(\dw-1)} \ge 1$ in the above equation, we obtain
\[
  p_t(x,y) \le  \frac{C_3}{t^{\df/\dw}} \exp\left(- \left(\frac{d(x,y)^\dw}{C_5 t}\right)^{1/(\dw-1)}\right).
\]
for all $x,y \in M$ and for all $t>0$ such that $d(x,y)^\dw \ge C_4 t$.

If $d(x,y)^\dw < C_4 t$, the on-diagonal estimate \eqref{e-diag} suffices to obtain the desired sub-Gaussian upper bound.
\qed
\begin{remark}
The following generalized capacity estimate is a  weaker form of the  cutoff Sobolev inequality \ref{csa}, where the cutoff function $\phi$  is allowed to depend on the function $f$. This generalized capacity estimate was introduced by Grigor'yan, Hu and Lau and they obtain sub-Gaussian estimate under this weaker assumption \cite[Theorem 1.2]{GHL2}.
\begin{definition} \label{d-gcap}
We say \ref{gcap} holds if there exists $C_1,C_2>0$ such that for every $x \in M$, $R > 0$, $r>0$, $f \in \F$, there exists a cutoff function $\phi = \phi(f)$ for  $B(x,R) \subset B(x,R+r)$ such that 
\begin{equation} \label{gcap}  \tag*{$\operatorname{Gcap}(\dw)$}
 \int_U f^2 \, d \Gamma(\phi,\phi) \le C_1 \int_U \phi^2 \, d\Gamma(f,f) + \frac{C_2}{r^\dw} \int_U f^2 \, d\mu,
\end{equation}
where $U= B(x,R+r) \setminus \overline{B(x,r)}$.
\end{definition}
We refer the reader to \cite[Definition 1.5]{CKW} for an analogous generalized capacity estimate in a non-local setting.

It is an interesting open problem to modify the proof of Theorem \ref{t-main} under the above weaker assumption. The main difficulty for carrying out Davies' method under the weaker generalized capacity estimate assumption is that we require the inequalities \eqref{e-sr1} and \eqref{e-sr2} as $t>0$ varies. This would require the cutoff function $\psi_k$  to depend on $f_{t,k}$ for each $t$. Therefore the derivatives computed in \eqref{e-sr1} and \eqref{e-sr2} will have additional terms, since $\psi_k$ varies with time $t$. 

We were informed of the reference \cite{HL} by the referee during the revision stage. In \cite{HL}, the techniques developed here and in the companion paper \cite{MS4} is extended to  Dirichlet forms on metric measure spaces (possibly non-local) with jumps satisfying a polynomial type upper bound.
\end{remark}
\subsection*{Acknowledgement}
We thank Martin Barlow for his enlightening remarks on the cutoff Sobolev inequalities  and for suggesting the proof of Lemma \ref{l-sip} that replaced a more complicated proof in an earlier draft. We thank Evan Randles for clarifying various aspects of Davies' work \cite{Dav2} on higher order operators.
We thank Tom Hutchcroft for proofreading part of the manuscript.
We thank the referee for  useful remarks and suggestions that led to an improved presentation.
\bibliographystyle{amsalpha}

\end{document}